\documentclass{amsart}
\usepackage{amssymb,amsfonts}
\usepackage[all,arc]{xy}
\usepackage{enumerate}
\usepackage{graphicx}
\usepackage{mathrsfs}
\usepackage{amsmath}
\usepackage{amsfonts}
\usepackage{hyperref}
\usepackage{cleveref}
\usepackage{tikz}
\usepackage{xypic}
\usepackage[utf8]{inputenc}
\usepackage[margin=1.3in]{geometry}

\newcommand{\R}{\mathbb{R}}
\newcommand{\C}{\mathbb{C}}

\newcommand{\Q}{\mathbb{Q}}

\newcommand{\Z}{\mathbb{Z}}

\newcommand{\ra}{\rightarrow}

\newcommand{\rdet}[1]{\operatorname{det}^{\operatorname{r}}_{\mathcal N#1}}
\newcommand{\tautwo}{\tau^{(2)}}
\newcommand{\slC}[1]{\operatorname{SL}(#1,\C)}

\newtheorem{thm}{Theorem}[section]
\newtheorem{cor}[thm]{Corollary}
\newtheorem{prop}[thm]{Proposition}

\newenvironment{ack}{{\noindent\textbf{Acknowledgement.}}\quad}

\newtheorem{lem}[thm]{Lemma}

\newtheorem{defn}[thm]{Definition}

\title{On the positivity of twisted $L^2$-torsion for 3-manifolds}
\author{Jianru Duan}

\begin{document}
\bibliographystyle{amsalpha}
\maketitle

\begin{abstract}
    For any compact orientable irreducible 3-manifold $N$ with empty or incompressible toral boundary, the twisted $L^2$-torsion is a non-negative function defined on the representation variety $\operatorname{Hom}(\pi_1(N),\slC n)$. The paper shows that if $N$ has infinite fundamental group, then the $L^2$-torsion function is strictly positive. Moreover, this torsion function is continuous when restricted to the subvariety of upper triangular representations.
\end{abstract}

\section{Introduction}
Let $N$ be a compact orientable irreducible 3-manifold with empty or incompressible toral boundary. The $L^2$-torsion of $N$ is a numerical topological invariant of $N$ that equals $\exp(\frac{\operatorname{Vol(N)}}{6\pi})$, where $\operatorname{Vol(N)}$ is the simplicial volume of $N$, see \cite[Theorem 4.3]{luck2002l2}. The idea of twisting is to use a linear representation of $\pi_1(N)$ to define more $L^2$-torsion invariants. 
The first attempt is made by Li and Zhang \cite{li2006l2, li2006alexander} in which they defined the $L^2$-Alexander invariants for knot complements, making use of the one dimensional representations of the knot group. Later Dubois, Friedl and L\"uck \cite{dubois2016l2} introduced the $L^2$-Alexander torsion for 3-manifolds which recovers the $L^2$-Alexander invariants. A recent breakthrough is made independently by Liu \cite{liu2017degree} and L\"uck \cite{luck2018twisting} that the $L^2$-Alexander torsion is always positive, and more interesting properties of the $L^2$-Alexander torsion are revealed in \cite{liu2017degree} and \cite{friedl2019l2}, for example, we now know that the  $L^2$-Alexander torsion is continuous and its limiting behavior recovers the Thurston norm of $N$.

Generally, let $\mathcal R_n(\pi_1(N)):=\operatorname{Hom}(\pi_1(N),\slC n)$ be the representation variety. One wishes to define $L^2$-torsion twisted by any representation $\rho\in\mathcal R_n(\pi_1(N))$, and we have this \emph{twisted $L^2$-torsion function} abstractly defined on the representation variety of $\pi_1(N)$:
\[\rho\longmapsto\tautwo(N,\rho)\in[0,+\infty),\quad \rho\in\mathcal R_n(\pi_1(N)).\]
A technical obstruction to defining a reasonable $L^2$-torsion is that the corresponding $L^2$-chain complex must be weakly $L^2$-acyclic and of determinant class (see definition \ref{DefinitionLtwoAcyclic}). If either condition is not satisfied, we define the $L^2$-torsion to be 0 by convention.

It is natural to question the positivity and continuity of this function. The first result of this paper is the following:
\begin{thm}\label{TheoremA}
Let $N$ be a compact orientable irreducible 3-manifold with empty or incompressible toral boundary. Suppose $N$ has infinite fundamental group, then the twisted $L^2$-torsion $\tautwo(N,\rho)$
is positive for any group homomorphism $\rho:\pi_1(N)\ra\slC n$.
\end{thm}

When $N$ is a graph manifold the twisted $L^2$-torsion function is explicitly computed in Theorem \ref{TorsionForGraph}. Other cases are dealt with in Theorem \ref{TorsionForHyper} where we only need to consider fibered 3-manifolds thanks to the virtual fibering arguments. We carefully construct a CW-structure for $N$ as in \cite{dubois2016l2} and observe that the matrices in the corresponding twisted $L^2$-chain complex are in a special form so that we can apply Liu's result \cite[Theorem 5.1]{liu2017degree} to guarantee the positivity of the Fuglede-Kadison determinant.

For continuity of the twisted $L^2$-torsion function, we have the following partial result:
\begin{thm}\label{TheoremB}
Let $N$ be a compact orientable irreducible 3-manifold with empty or incompressible toral boundary. Suppose $N$ has infinite fundamental group. Define  $\mathcal R^{\operatorname{t}}_n(\pi_1(N))$ to be the subvariety of $\mathcal R_n(\pi_1(N))$ consisting of upper triangular representations. Then the twisted $L^2$-torsion function
\[\rho\longmapsto \tautwo(N,\rho)\]
is continuous with respect to $\rho\in \mathcal R^{\operatorname{t}}_n(\pi_1(N))$.

\end{thm}

The continuity of the twisted $L^2$-torsion function in general is open. It is mainly because the Fuglede-Kadison determinant of an arbitrary matrix over $\C[\pi_1(N)]$ is very difficult to compute. However, the $L^2$-torsion twisted by upper triangular representations are relatively simpler because we can reduce many problems to the one-dimensional case, which is well studied under the name of the $L^2$-Alexander torsion (see \cref{section 5}). We remark that the work of Benard and Raimbault \cite{benard2022twisted} based on the strong acyclicity property by Bergeron and Venkatesh \cite{bergeron2013asymptotic} shows that the twisted $L^2$-torsion function is positive and real analytic near any holonomy representation $\rho_0:\pi_1(N)\ra \slC2$ of a hyperbolic 3-manifold $N$.

The proof relies on the continuity of $L^2$-Alexander torsion with respect to the cohomology classes, which is conjectured by \cite[Chapter 10]{luck2018twisting}. This is done by introducing the concept of Alexander multi-twists (see section \ref{section 5}). One can similarly define the ``multi-variable $L^2$-Alexander torsion" and our argument essentially shows that the multi-variable function is multiplicatively convex (compare Theorem \ref{MultivariableConvexityOfDeterminant}), generalizing \cite[Theorem 5.1]{liu2017degree}. This then applies to show the continuity as desired.

The organization of this paper is as follows. In section \ref{section 2}, we introduce the terminology of this paper and some algebraic facts. In section \ref{section 3}, we define the twisted $L^2$-torsion for CW complexes and state some basic properties. In section \ref{section 4} we prove Theorem \ref{TheoremA} in two steps: first for graph manifold, then for hyperbolic or mixed manifold. In section \ref{section 5}, we begin with the $L^2$-Alexander torsion and then prove Theorem \ref{TheoremB}. ~\\

\begin{ack}
The author wishes to thank his advisor Yi Liu for guidance and many conversations.
\end{ack}

\section{Notations and some algebraic facts}\label{section 2}

In this section we define the twisting functor and introduce $L^2$-torsion theory. The reader can refer to \cite{luck2018twisting} where discussions are taken on in a more general setting.

\subsection{Twisting $\C G$-modules via $\slC n$ representations}

Let $G$ be a finitely generated group and let $\C G$ be its group ring. In this paper our main objects are finitely generated free left $\C G$-modules with a preferred ordered basis. We will abbreviate it as \emph{based $\C G$-modules} unless otherwise stated. A natural example of a based $\C G$-module is $\C G^m$ as a free left $\C G$-module of rank $m$, with the natural ordered basis $\{\sigma_1,\cdots,\sigma_m\}$ where $\sigma_i$ is the unit element of the $i^{\operatorname{th}}$ direct summand. Any based $\C G$-module is canonically isomorphic to $\C G^m$ for some non-negative integer $m$ and this identification is used throughout our paper.

We fix $V$ throughout this paper to be an $n$-dimensional complex vector space with a fixed choice of basis $\{e_i\}_{i=1}^n$. Let $\rho:G\ra \slC n$ be a group homomorphism, then $V$ can be viewed as a left $\C G$-module via $\rho$, in the following way:
\[\gamma\cdot e_i=\sum_{j=1}^n\rho(\gamma^{-1})_{i,j}\cdot e_j,\quad \gamma\in G\]
where $\rho(\gamma^{-1})\in \slC n$ is a square matrix. We extend this action $\C$-linearly so that $V$ is a left $\C G$-module. In other words, left action of $\gamma$ corresponds to right multiplication to the row coordinate vector of the matrix $\rho(\gamma^{-1})$. 

We are interested in twisting a based $\C G$-module via $\rho$. In literature, there are two different ways to twist a based $\C G$-module, namely the ``diagonal twisting" and the ``partial twisting"  (compare \cite{luck2018twisting}). They are naturally isomorphic. We only consider the diagonal twisting in this paper.
\begin{defn}
Recall that $\C G^m$ is a based $\C G$-module with a natural basis $\{\sigma_i\},\ i=1,\cdots,m$. We define $(\C G^m\otimes_\C V)_{\operatorname{d}}$ to be the  $\C G$-module with diagonal $\C G$-action, i.e.
\[(\C G^m\otimes_\C V)_{\operatorname{d}}\stackrel{\operatorname{set}}=\C G^m\otimes_\C V,\quad g\cdot(u\otimes v)=gu\otimes gv\]
for any $g\in G,\ u\in\C G^m$ and $v\in V$, and then extend $\C$-linearly to define a $\C G$-module structure. 
\end{defn}

With the definition above, we can see that
\[(\C G^m\otimes_\C V)_{\operatorname{d}}=\bigoplus_{i=1}^m(\C G\otimes_\C V)_{\operatorname{d}}\]
is a based $\C G$-module with a basis
\[\{\sigma_1\otimes e_1, \sigma_1\otimes e_2,\cdots,\sigma_1\otimes e_n,\sigma_2\otimes e_1,\cdots,\sigma_m\otimes e_n\}.\]

Let $\mathcal A$ be the category whose objects are finitely generated free left $\C G$-modules with a preferred ordered basis and whose morphisms are $\C G$-linear homomorphisms. We consider the following ``\emph{diagonal twisting}" functor
\[\mathcal D(\rho):\mathcal A\longrightarrow\mathcal A\]
which sends any object $M$ to the based $\C G$-module $(M\otimes_{\C} V)_{\operatorname{d}}$ and sends any morphism $f$ to $\mathcal D(\rho)f:=f\otimes_\C \operatorname{id_V}$. The following proposition describes how matrices behave under the twisting functor. 

\begin{prop}\label{MatrixRepPartialTwist}
Let $\rho:G\ra \slC n$ be any group homomorphism. Suppose that a homomorphism between based $\C G$-modules
\[f:\C G^{r}\longrightarrow \C G^{s}\]
is presented by a matrix $(\Lambda_{i,j})$ over $\C G$ of size $r\times s$, i.e., let \[\{\sigma_1,\cdots,\sigma_{r}\},\quad \{\tau_1,\cdots,\tau_{s}\}\] be the natural basis of $\C G^{r}$ and $\C G^{s}$ respectively, we have
\[f(\sigma_i)=\sum_{j=1}^{s} \Lambda_{i,j}\tau_j,\quad i=1,\cdots,r.\]
We form a new matrix $\Omega$ of size $nr\times ns$ by replacing each entry $\Lambda_{i,j}$ with an $n\times n$ square matrix $\Lambda_{i,j}\cdot \rho(\Lambda_{i,j})$. Then $\Omega$ is a matrix presenting the diagonal twisting morphism $\mathcal D(\rho)f$, under the natural basis
\[\{\sigma_1\otimes e_1,\cdots,\sigma_1\otimes e_n,\sigma_2\otimes e_1,\cdots,\sigma_{r}\otimes e_n\},\]
\[\{\tau_1\otimes e_1,\cdots,\tau_1\otimes e_n,\tau_2\otimes e_1,\cdots,\tau_{s}\otimes e_n\}\]
of the diagonal twisting based $\C G$-modules $\mathcal D(\rho)(\C G^r)$ and $\mathcal D(\rho)(\C G^s)$ respectively.
\end{prop}

\begin{proof}
Let $\Phi=(\Phi_{i,j}),\ i=1,\cdots,r,\ j=1,\cdots,s$ be a block matrix of size $nr\times ns$, with each entry $\Phi_{i,j}$ an $n\times n$ matrix, such that $\Phi$ is the matrix presenting $\mathcal D(\rho)f$ under the natural basis. We only need to verify that $\Phi_{i,j}=\Lambda_{i,j}\cdot\rho(\Lambda_{i,j})$. The submatrix $\Phi_{i,j}$ can be characterized as follows. Let $\pi_j:\mathcal D(\rho)(\C G^r)\ra \mathcal D(\rho)(\C G)$ be the projection to the $j^{\operatorname{th}}$ direct component which is spanned by $\{(\sigma_j\otimes e_1)_{\operatorname{d}},\cdots,(\sigma_j\otimes e_n)_{\operatorname{d}}\}$. Then the following holds:
\[\pi_j\circ \mathcal D(\rho)f
\begin{pmatrix}
(\sigma_i\otimes e_1)_{\operatorname{d}} \\ \vdots \\ (\sigma_i\otimes e_n)_{\operatorname{d}}
\end{pmatrix}=\Phi_{i,j}
\begin{pmatrix}
(\tau_j\otimes e_1)_{\operatorname{d}} \\ \vdots \\ (\tau_j\otimes e_n)_{\operatorname{d}}
\end{pmatrix}.\]
On the other hand, for any $k=1,\cdots,n$, we have
\[
\begin{aligned}
\pi_j\circ \mathcal D(\rho)f((\sigma_i\otimes e_k)_{\operatorname{d}})&=\pi_j\left(\sum_{l=1}^{s} (\Lambda_{i,l}\tau_l\otimes e_k)_{\operatorname{d}}\right)\\
&=\pi_j\left(\sum_{l=1}^{s} \Lambda_{i,l}\cdot(\tau_l\otimes \Lambda_{i,l}^{-1}e_k)_{\operatorname{d}}\right)\\
&=\Lambda_{i,j}\cdot(\tau_j\otimes \Lambda_{i,j}^{-1}e_k)_{\operatorname{d}}\\
&=\Lambda_{i,j}\cdot\sum_{l=1}^n\rho(\Lambda_{i,j})_{k,l}(\tau_j\otimes e_l)_{\operatorname{d}}.
\end{aligned}
\]
This shows that $\Phi_{i,j}=\Lambda_{i,j}\cdot\rho(\Lambda_{i,j})$ and hence $\Phi=\Omega$.
\end{proof}

We now mention that the twisting functor can be naturally generalized to the category of \emph{based $\C G$-chain complexes}. More explicitly, let $C_*$ be a based $\C G$-chain complex, i.e.
\[C_*=(\cdots\longrightarrow C_{p+1}\stackrel{\partial_{p+1}}\longrightarrow C_p \stackrel{\partial_{p}}\longrightarrow C_{p-1} \longrightarrow \cdots)\]
is a chain of based $\C G$-modules with $\C G$-linear connecting morphisms $\{\partial_p\}$ such that $\partial_{p-1}\circ \partial_p=0$. We can apply the functor $\mathcal D(\rho)$ to obtain a new $\C G$-chain complex
\[\mathcal D(\rho) C_*=(\cdots\longrightarrow \mathcal D(\rho)C_{p+1}\stackrel{\mathcal D(\rho)\partial_{p+1}}\longrightarrow \mathcal D(\rho)C_p \stackrel{\mathcal D(\rho)\partial_{p}}\longrightarrow \mathcal D(\rho)C_{p-1} \longrightarrow \cdots)\]
with connecting homomorphisms $\{\mathcal D(\rho)\partial_p\}$.
If $f_*$ is a chain map between based $\C G$-chain complexes, the twisting chain map $\mathcal D(\rho)f_*$ is a $\C G$-chain map between the corresponding twisted chain complexes. So $\mathcal D(\rho)$ generalizes to be a functor of the category of based $\C G$-chain complexes.

\subsection{$L^2$-torsion theory}
Let 
\[l^2(G)=\Big\{\sum_{g\in G}c_g\cdot g\ \Big|\ c_g\in\C,\  \sum_{g\in G}|c_g|^2<\infty\Big\}\]
be the Hilbert space orthonormally spanned by all elements in $G$. Since $G$ is finitely generated, $l^2(G)$ is a separable Hilbert space with isometric left and right $\C G$-module structure. We denote by $\mathcal N(G)$ the \emph{group von Neumann algebra} of $G$ which consists of all bounded Hilbert operators of $l^2(G)$ that commute with the right $\C G$-action. We will treat $l^2(G)$ as a left $\mathcal N(G)$-module and a right $\C G$-module. The \emph{$l^2$-completion} of a based $\C G$-chain complex $C_*$ is then a \emph{Hilbert $\mathcal N(G)$-chain complex} defined as 
\[l^2(G)\otimes_{\C G}C_*\]
and the $l^2$-completions of the connecting homomorphism $\partial$ and chain map $f$ are $\operatorname{id}\otimes_{\C G} \partial$ and $\operatorname{id}\otimes_{\C G} f$ respectively. Note that each chain module of $l^2(G)\otimes_{\C G}C_*$ is simply a direct sum of $l^2(G)$:
\[l^2(G)\otimes_{\C G}C_p=l^2(G)\otimes_{\C G}\C G^{r_p}=l^2(G)^{r_p}\]
where $r_p$ is the rank of $C_p$.

The $l^2$-completion process converts a based $\C G$-chain complex into a finitely generated, free Hilbert $\mathcal  N(G)$-chain complex.

\begin{defn}\label{DefinitionLtwoAcyclic}
A finitely generated, free Hilbert $\mathcal N (G)$-chain complex is called \emph{weakly acyclic} if the $l^2$-Betti numbers are all trivial. A finitely generated, free Hilbert $\mathcal N (G)$-chain complex is \emph{of determinant class} if all the Fuglede-Kadison determinants of the connecting homomorphisms are positive real numbers.
\end{defn}

\begin{defn}
Let $C_*$ be a finitely generated, free Hilbert $\mathcal N (G)$-chain complex. Suppose $C_*$ is of finite length, i.e., there exists an integer $N>0$ such that $C_p=0$ for $|p|>N$. Furthermore, if $C_*$ is weakly acyclic and of determinant class, we define the \emph{$L^2$-torsion} of $C_*$ to be the alternating product of the Fuglede-Kadison determinants of the connecting homomorphisms:
\[\tau^{(2)}(C_*)=\prod_{p\in\Z}(\operatorname{det}_{\mathcal N(G)}\partial_p)^{(-1)^p}.\]
Otherwise, we artificially set $\tautwo(C_*)=0$.
\end{defn}

We recommend \cite{luck2002l2} for the definition of the $L^2$-Betti number and the Fuglede-Kadison determinant. We remark that our notational convention follows \cite{dubois2015l2,dubois2016l2,liu2017degree}, and the exponential of the torsion in \cite{luck2002l2,luck2018twisting} is the multiplicative inverse of our torsion.

Let $A$ be a $p\times p$ matrix over $\mathcal N(G)$. The \emph{regular Fuglede-Kadison determinant} of $A$ is defined to be 
\[\rdet G(A)=\left\{
\begin{array}{ll}
\operatorname{det}_{\mathcal{N}(G)}(A),     & \text{if $A$ is full rank of determinant class,}  \\
0,     &  \text{otherwise.}
\end{array}\right.
\]

We will need the following two lemmas in order to do explicit calculations, the proof can be found in \cite[Lemma 2.6, Lemma 3.2]{dubois2015l2} combining with the basic properties of the Fuglede-Kadison determinant (see \cite[Theorem 3.14]{luck2002l2}).

\begin{lem}\label{MahlerMasureFormula}
Let $\Z^k$ be a free Abelian subgroup of $G$ generated by $z_1,\cdots,z_k$. Let $A$ be a $p\times p$ matrix over $\C \Z^k$. Identify $\C \Z^k$ with the $k$-variable Laurent polynomial ring $\C[z_1^{\pm},\cdots,z_k^{\pm}]$. Denote by $p(z_1,\cdots,z_k)$ the ordinary determinant of $A$, then 
\[\rdet G(A)=\operatorname{Mah}(p(z_1,\cdots,z_k))\]
where $\operatorname{Mah}(p(z_1,\cdots,z_k))$ is the Mahler measure of the polynomial $p(z_1,\cdots,z_k)$.
\end{lem}
\begin{lem}
\label{MatrixCalculationTorsion}
Let
\[D_*=(0\longrightarrow \C G^j\stackrel{C}\longrightarrow \C G^k \stackrel{B}\longrightarrow \C G^{k+l-j} \stackrel{A}\longrightarrow\C G^l \longrightarrow 0 )\]
be a complex.  Let $L\subset\{1,\cdots, k + l - j\}$ be a subset of size $l$ and $J \subset \{1,\cdots, k\}$ a
subset of size $j$. We write
\[A(J):=\text{rows in $A$ corresponding to $J$.}\]
\[
\begin{aligned}
    B(J,L):=&\text{ result of deleting the columns of $B$ corresponding to $J$}\\
    &\text{ and deleting the rows corresponding to $L$.}
\end{aligned}
\]
\[C(J):=\text{columns of $C$ corresponding to $L$.}\]

View $A,B,C$ as matrices over $\mathcal N(G)$. If $\rdet G(A(J))\not=0$ and $\rdet G(C(L))\not=0$, then
\[\tau^{(2)}(l^2(G)\otimes_{\C G}D_*)=\rdet G(B(J,L))\cdot\rdet G(A(J))^{-1}\cdot\rdet G(C(L))^{-1}.\]
\end{lem}

\section{Twisted $L^2$-torsion for CW complexes}\label{section 3}
Let $X$ be a finite CW complex with fundamental group $G$. Denote by $\widehat X$ the universal cover of $|X|$ with the natural CW complex structure coming from $X$. Choose a lifting $\widehat\sigma_i$ for each cell $\sigma_i$ in the CW structure of $X$. The deck group $G$ acts freely on the cellular chain complex of $\widehat X$ on the left, which makes the $\C$-coefficient cellular chain complex $C_*(\widehat X)$ a based $\C G$-chain complex with basis $\{\widehat\sigma_i\}$. Recall that $\rho:G\ra \slC n$ is any group homomorphism. 

For future convenience, we introduce the concept of \emph{admissible triple} for higher dimensional linear representations, generalizing the admissibility condition in \cite{dubois2015l2}.
\begin{defn}[Admissible triple]
Let $\gamma:G\ra H$ be a homomorphism to a countable group $H$. We say that $(G,\rho;\gamma)$ forms an \emph{admissible triple} if $\rho:G\ra\slC n$ factors through $\gamma$, i.e., for some homomorphism $\psi:H\ra \slC n$, the following diagram commutes:
\[\xymatrix{
G \arrow[r]^\gamma \arrow[rd]_\rho & H \arrow[d]^\psi \\
& \slC n
}\]
\end{defn}

\begin{defn}
Let $(G,\rho;\gamma)$ be an admissible triple. Consider $l^2(H)$ as a left Hilbert $\mathcal N(H)$-module, and a right $\C G$-module induced by $\gamma$. Define the $L^2$-chain complex of $X$ twisted by $(G,\rho;\gamma)$ to be the following Hilbert $\mathcal{N}(H)$-chain complex
\[C^{(2)}_*(X,\rho;\gamma):=l^2(H)\otimes_{\C G} \mathcal D(\rho)C_*(\widehat X).\]
We define the $L^2$-torsion of $X$ twisted by $(G,\rho;\gamma)$ as
\[\tautwo(X,\rho;\gamma):=\tautwo(C^{(2)}_*(X,\rho;\gamma)).\]
\end{defn}
\begin{prop}\label{InvarianceOfBasis}
The definition of $\tautwo(X,\rho;\gamma)$ with respect to any admissible triple $(G,\rho;\gamma)$ does not depend on the order or orientation of the basis $\{\sigma_i\}$, nor the choice of lifting $\{\widehat\sigma_i\}$. Moreover, let $\rho':G\ra \slC n$ be conjugate to $\rho$, i.e., there exists a matrix $T\in\slC n$, such that $\rho'=T\cdot \rho\cdot T^{-1}$. Then $(G,\rho';\gamma)$ is also an admissible triple and $\tautwo(X,\rho;\gamma)=\tautwo(X,\rho';\gamma)$.
\end{prop}
\begin{proof}
The property of being weakly $L^2$-acyclic does not depend on the choices in the statement. We only need to analyze how these choices change the Fuglede-Kadison determinant of the connecting morphisms.

Abbreviate by $C_*(\widehat X,\rho):=\mathcal D(\rho)C_*(\widehat X;\C)$ the diagonal twisting chain complex.
Suppose the based cellular chain complex of $\widehat X$ has the form
\[C_*(\widehat X)=(\cdots\longrightarrow \C G ^{r_{i+1}} \stackrel{\partial_{i+1}}\longrightarrow \C G ^{r_{i}} \stackrel{\partial_i}\longrightarrow \C G ^{r_{i-1}}\longrightarrow\cdots)\]
where $\partial_i$ is an $r_i\times r_{i-1}$ matrix over $\C G$ for all $i$, 
then the diagonal twisting chain complex $C_*(\widehat X,\rho)$ has the form
\[C_*(\widehat X,\rho)=(\cdots\longrightarrow \C G ^{nr_{i+1}} \stackrel{\partial^\rho_{i+1}}\longrightarrow \C G ^{nr_{i}} \stackrel{\partial_i^\rho}\longrightarrow \C G ^{nr_{i-1}}\longrightarrow\cdots)\]
where $\partial^\rho_i=\mathcal D(\rho)\partial_i$ is an $nr_i\times nr_{i-1}$ matrix over $\C G$ for all $i$. An explicit formula for $\partial^\rho_i$ is presented in Proposition \ref{MatrixRepPartialTwist}.
Then the $L^2$-chain complex of $X$ twisted by $(G,\rho;\gamma)$ has the form
\[C^{(2)}_*(X,\rho;\gamma)=(\cdots\longrightarrow l^2(H)^{nr_{i+1}} \stackrel{\gamma(\partial^\rho_{i+1})}\longrightarrow l^2(H)^{nr_{i}} \stackrel{\gamma(\partial_i^\rho)}\longrightarrow l^2(H)^{nr_{i-1}}\longrightarrow\cdots),\]
the notation $\gamma(\partial^\rho_i)$ means applying the group homomorphism $\gamma$ to each monomial of any entry of the matrix $\partial^\rho_i$, resulting in a matrix over $\C H\subset\mathcal N(H)$.

We now analyze how the choices affect the value of $\tautwo(X,\rho;\gamma)$. 
If the basis of $C_i(X)$ is permuted, and the orientations are changed, then $\gamma(\partial^\rho_i)$ and $\gamma(\partial^\rho_{i+1})$ change by multiplying a permutation matrix, with entries $\pm1$.

If one choose another lifting $g\widehat\sigma$ instead of $\widehat\sigma$ for some $g\in G$, then $\gamma(\partial^\rho_i)$ and $\gamma(\partial^\rho_{i+1})$ change by multiplying a block matrix in the following form:
\[\begin{pmatrix}
I^{n\times n} & & & &\\
 & \ddots & & & \\
  & & \rho(g)^{\pm1}\cdot I^{n\times n} & &\\
  & & & \ddots &\\
 & & & & I^{n\times n}
\end{pmatrix}.\]

If one replace $\rho$ by $\rho'=T\cdot \rho \cdot T^{-1}$ for a matrix $T\in \slC n$, the corresponding connecting homomorphism is in the following form:
\[\gamma(\partial^{\rho'}_i)=
\begin{pmatrix}
T & & \\
 & \ddots &  \\
  & & T \\
\end{pmatrix}
\gamma(\partial_i^\rho)
\begin{pmatrix}
T^{-1} & & \\
 & \ddots &  \\
  & & T^{-1} \\
\end{pmatrix}\]

In all cases, the regular Fuglede-Kadison determinant of $\gamma(\partial^\rho_i)$ and $\gamma(\partial^\rho_{i+1})$ are unchanged by basic properties of Fuglede-Kadison determinant, see \cite[Theorem 3.14]{luck2002l2}.
\end{proof}

Note that the ``moreover" part of the previous lemma tells us that we don't need to worry about the different choices of the base point when identifying the fundamental group $\pi_1(X)$ with $G$.

\begin{lem}\label{TorusTorsionFormula}
Let $T$ be a two-dimensional torus. For any admissible triple $(T,\rho:\pi_1(T)\ra\slC n;\gamma:\pi_1(T)\ra H)$, if $\operatorname{im} \gamma$ is infinite, then
\[\tautwo(T,\rho;\gamma)=1.\]
\end{lem}
\begin{proof}
We consider the standard CW structure for $T$ constructed by identifying pairs of sides of a square. Let $P$ be the 0-cell. Let $E_1,E_2$ be the 1-cells. Let \[e_1=[E_1]\in\pi_1(T),\quad e_2=[E_2]\in \pi_1(T),\]
then $\pi_1(T)$ is the free Abelian group generated by $e_1,e_2$. There is a 2-cell $\sigma$ whose boundary is the loop $E_1E_2E_1^{-1}E^{-1}_2$. Let $\widehat T$ be the universal covering of $T$ with the induced CW structure. It is easy to see that the $L^2$-chain complex of $T$ twisted by $(\pi_1(T),\rho;\gamma)$ is
\[C^{(2)}_*(T,\rho;\gamma)=(0\longrightarrow l^2(H)\langle\sigma\rangle\otimes_\C V\stackrel{\gamma(\partial_2^\rho)}\longrightarrow l^2(H)\langle E_1,E_2\rangle\otimes_\C V \stackrel{\gamma(\partial_1^\rho)}\longrightarrow l^2(H)\langle P\rangle\otimes_\C V \longrightarrow0)\]
in which
\[\gamma(\partial^\rho_2)=\begin{pmatrix}
I^{n\times n}-\gamma(e_2)\rho(e_2) & 
-I^{n\times n}+\gamma(e_1)\rho(e_1)
\end{pmatrix},\quad
\gamma(\partial_1^\rho)=
\begin{pmatrix}
\gamma(e_1)\rho(e_1)-I^{n\times n} \\ \gamma(e_2)\rho(e_2)-I^{n\times n}\\
\end{pmatrix}.
\]
We assume without loss of generality that $\gamma(e_1)$ has infinite order. Set $p(z):=\det(z\rho(e_1)-I^{n\times n})$ as a polynomial of indeterminant $z$. Then by Lemma \ref{MahlerMasureFormula} \[\rdet H(\gamma(e_1)\rho(e_1)-I^{n\times n})=\operatorname{Mah}(p(z))\not=0.\]
The conclusion follows from \cite[Lemma 3.1]{dubois2015l2} which is a formula analogous to Lemma \ref{MatrixCalculationTorsion} but applies to shorter chain complexes.
\end{proof}

There is another way to define the twisted $L^2$-torsions, following L\"uck \cite{luck2018twisting}. Let $H$ be a finitely generated group. Recall that $\widetilde X$ is called a \emph{finite free $H$-CW complex} if $\widetilde X$ is a regular covering space of a finite CW complex $X$, with deck transformation group $H$ acting on $\widetilde X$ on the left. Choose an $H$-equivariant CW structure for $\widetilde X$, and choose one representative cell for each $H$-orbit, then the cellular chain complex $C_*(\widetilde X)$ becomes a based $\C H$-chain complex. For any group homomorphism $\phi:H\ra \slC n$, we form the diagonal twisting chain complex $\mathcal D(\phi)C_*(\widetilde X)$ (recall the definition of the twisting functor $\mathcal D$ in section \ref{section 2}). The \emph{$\phi$-twisted $L^2$-torsion} of the $H$-CW complex $\widetilde X$ is defined to be
\[\rho_H^{(2)}(\widetilde X,\phi):=\log\tautwo(l^2(H)\otimes_{\C H}\mathcal D(\phi)C_*(\widetilde X)).\]
Note that $\rho$ is a unimodular representation in our setting, this torsion does not depend on a specific $\C H$-basis for $C_*(\widetilde X)$ (compare Proposition \ref{InvarianceOfBasis}). We point out in the following proposition that both definitions of twisted $L^2$-torsion are essentially the same.

\begin{prop}
Following the notations above. Let $G$ be the fundamental group of $X=H\backslash\widetilde X$, there is a natural quotient map $\gamma:G\ra H$ by covering space theory. It is obvious that $(G,\phi\circ \gamma;\gamma)$ is an admissible triple. Then we have
\[\tautwo(X,\phi\circ\gamma;\gamma)=\exp\rho^{(2)}_H(\widetilde X,\phi).\]
\end{prop}
\begin{proof}
Let $\widehat X$ be the universal covering space of $X$, with the natural CW structure coming from $X$. Choose a lifting for each cell in $X$ and then $C_*(\widehat X)$ becomes a based $\C G$-chain complex. It is a pure algebraic fact that the two based $\C H$-chain complexes are $\C H$-isomorphic:
\begin{equation}\tag{*}\label{AlgebraicFact} \mathcal D(\phi)C_*(\widetilde X)\cong \C H\otimes_{\C G} \mathcal D(\phi\circ\gamma)C_*(\widehat X).\end{equation}
Indeed, the $\C H$-chain complex $\C H\otimes_{\C G} \mathcal D(\phi\circ\gamma)C_*(\widehat X)$ is obtained from \[C_*(\widehat X)=(\cdots\longrightarrow \C G ^{r_{i+1}} \stackrel{\partial_{i+1}}\longrightarrow \C G ^{r_{i}} \stackrel{\partial_i}\longrightarrow \C G ^{r_{i-1}}\longrightarrow\cdots)\] by the following two operations:

(1) (the diagonal twist) firstly, replace every direct summand $\C G$ by its $n^{\operatorname{th}}$ power $\C G^n$, replace any entry $\Lambda_{i,j}$ of the matrix $\partial_*$ by a block matrix $\Lambda_{i,j}\phi\circ\gamma(\Lambda_{i,j})$, as in Proposition \ref{MatrixRepPartialTwist}, resulting in a new matrix $\partial_*^{\phi\circ \gamma}$ and then

(2) (tensoring with $\C H$) replace every direct summand $\C G$  of the chain module by $\C H$, and apply $\gamma$ to every entry of  $\partial_*^{\phi\circ \gamma}$, resulting in a block matrix whose $i,j$-submatrix is $\gamma(\Lambda_{i,j})\phi\circ\gamma(\Lambda_{i,j})$.

The resulting chain complex is exactly the chain complex $\mathcal D(\phi)(\C H\otimes_{\C G} C_*(\widehat X))$ (this can be seen by doing the above operations in the reversed order, thanks to the admissible condition). Combining the well-known $\C H$-isomorphism $C_*(\widetilde X)\cong\C H\otimes_{\C G} C_*(\widehat X)$ and then the isomorphism \eqref{AlgebraicFact} follows.

Finally, we tensor $l^2(H)$ on the left of both $\C H$-chain complexes and the conclusion follows from both taking $L^2$-torsion.
\end{proof}

The following useful properties are obtained by translating the statements of \cite[Theorem 6.7]{luck2018twisting} into our terminology.

\begin{lem}\label{BasicProppertyTorsion}
Some basic properties of twisted-$L^2$ torsions:

(1) $G$-homotopy equivalence.

Let $X,Y$ be two finite CW complexes with fundamental group $G$. For any admissible triple $(G,\rho;\gamma)$, suppose there is a simple homotopy equivalence $f:X\ra Y$ such that the induced homomorphism $f_*:G\ra G$ preserves $\ker \gamma$. Then we have
\[\tautwo(X,\rho;\gamma)=\tautwo(Y,\rho;\gamma).\]

(2) Restriction.

Let $X$ be a finite CW complex with fundamental group $G$. Let $\widetilde X$ be a finite regular cover of $X$ with the induced CW structure. Suppose $\pi_1(\widetilde X)=\widetilde G\lhd G$ is a normal subgroup of index $d$. Let $\widetilde \rho:\widetilde G\ra \slC n$ be the restriction of $\rho: G\ra \slC n$. Then
\[\tautwo(\widetilde X,\widetilde \rho)=\tautwo(X,\rho)^d.\]

(3) Sum formula.

Let $X$ be a finite CW complex with fundamental group $G$ and $\rho:G\ra \slC n$ be a homomorphism. Let 
\[i_1:X_1\hookrightarrow X, \ i_2:X_2\hookrightarrow X, \ i_0:X_1\cap X_2\hookrightarrow X\] be subcomplex of $X$ 
with $X_1\cup X_2=X$. Let \[\rho_1=\rho|_{\pi_1(X_1)},\ \rho_2=\rho|_{\pi_1(X_2)},\ \rho_0=\rho|_{\pi_1(X_1\cap X_2)}\]
be the restriction of $\rho$.
If $\tautwo(X_1\cap X_2,\rho_0;i_{0*})\not=0$, then
\[\tautwo(X,\rho)=\tautwo(X_1,\rho_1;i_{1*})\cdot\tautwo(X_2,\rho_2;i_{2*})/\tautwo(X_1\cap X_2,\rho_0;i_{0*}).\]
\end{lem}

\section{Twisted $L^2$-torsion for 3-manifolds}\label{section 4}
In the remaining of this paper, we will assume that $N$ is a compact orientable irreducible 3-manifold with empty or incompressible toral boundary. We denote by $G$ the fundamental group of $N$ and assume $G$ is infinite. It is well known that $G$ is finitely generated and residually finite (see  \cite{hempel1987residual}). For any group homomorphism $\rho:G\ra \slC n$ and $\gamma:G\ra H$, we say $(N,\rho;\gamma)$ is an admissible triple if $(G,\rho;\gamma)$ is. In this case, we define the \emph{twisted $L^2$-torsion of $(N,\rho;\gamma)$} by
\[\tautwo(N,\rho;\gamma):=\tautwo(X,\rho;\gamma)\]
where $X$ is any CW structure for $N$. This definition does not depend on the choice of $X$, thanks to Lemma \ref{BasicProppertyTorsion}. Indeed, if $X, Y$ are two CW structures for $N$, denote by $f:X\ra Y$ the corresponding homeomorphism, then $f$ is a simple homotopy equivalence by Chapman \cite[Theorem 1]{chapman1974topological} and certainly preserves $\ker \gamma$. So we have $\tautwo(X,\rho;\gamma)=\tautwo(Y,\rho;\gamma)$.

The remaining part of this section is devoted to the proof of Theorem \ref{TheoremA}.

\subsection{Twisted $L^2$-torsion for graph manifolds}
We prove Theorem \ref{TheoremA} for graph manifold $N$ with infinite fundamental group $G$.
\begin{thm}\label{TorsionForGraph}
Suppose $M$ is a Seifert-fibered piece of the graph manifold $N$. Let $h\in \pi_1(M)$ be represented by the regular fiber of $M$. Let $\Lambda$ be the product of all eigenvalues of $\rho(h)$ whose modulus is not greater than 1. Suppose the orbit space $M/S^1$ has orbifold Euler characteristic $\chi_{\operatorname{orb}}$. Then
\[\tautwo(N,\rho)=\prod_{M\subset  N \text{ is a  Seifert piece}}
\Lambda^{\chi_{\operatorname{orb}}}\]
\end{thm}
\begin{proof}
This proof is a generalization of \cite[Proposition 4.3]{benard2022twisted}. Fix any Seifert-fibered piece $M$ of the JSJ-decomposition of $N$, then $\pi_1(M)$ is infinite as well. Suppose that $M$ is isomorphic to a model
\[M(g,b;q_1/p_1,\cdots,q_k/p_k),\quad k\geqslant1,\ p_1\cdots,p_k>0\]
following Hatcher \cite{hatcher2007notes}, more explicitly, take a surface of genus $g$ with $b$ boundary components, namely $E_1,\cdots,E_b$, then drill out $k$-disjoint disks from it to form a new surface $\Sigma$ with $k$ additional boundary circles $F_1,\cdots,F_k$. These $k$ boundary circles correspond to $k$ boundary tori of $\Sigma\times S^1$, namely $T_1,\cdots,T_k$, then $M$ is obtained by a Dehn filling of slope $(q_1/p_1,\cdots,q_k/p_k)$ along $(T_1,\cdots,T_k)$ respectively. So we have
\[M=(\Sigma\times S^1)\cup_{T_1}D_1\cup_{T_2}\cdots\cup_{T_k}D_k\]
in which $D_i$ is a solid torus whose meridian $(0,1)$-curve is attached to the $(q_i,p_i)$-curve of $T_i$. The orbit space can be viewed as a 2-dimensional orbifold, whose underlying topological space is a surface $\Sigma_{g,b}$ with $k$ singularities of indices $p_1,\cdots,p_k$ respectively. The orbifold Euler characteristic is 
\[\chi_{\operatorname{orb}}=2-2g-b-\sum_{i=1}^k(1-\frac1{p_i}).\]
More details can be found in \cite{scott1983geometries}.

Retract $\Sigma$ along the boundary circle $F_k$ to an $1$-dimensional complex $X$, it is a bunch of circles with one common vertex $P$, and edges 
\[A_1,B_1,\cdots,A_g,B_g,E_1,\cdots,E_b,F_1,\cdots,F_{k-1}\]
where $A_1,B_1,\cdots,A_g,B_g$ come from the standard polygon representation of a closed surface $\Sigma_g$.
Suppose that $A_i,B_i,E_i,F_i$ represents $a_i,b_i,e_i,f_i$ in $\pi_1(M)$ respectively. Let $H$ be the 1-cell of $S^1$ representing $h\in\pi_1(M)$, then $\Sigma\times S^1$ is given the product CW structure, we collect the cells in each dimension in the following:
\[\{A_1\times H,B_1\times H,\cdots,A_g\times H,B_g\times H,E_1\times H,\cdots,E_b\times H,F_1\times H,\cdots,F_{k-1}\times H\},\]
\[\{A_1,B_1,\cdots,A_g,B_g,E_1,\cdots,E_b,F_1,\cdots,F_{k-1},H\},\quad \{P\}.\]
We have $f_i^{p_i}h^{q_i}=1$ for $i=1,\cdots,k-1$ by the Dehn filling.

Denote by \[\kappa:\Sigma\times S^1\hookrightarrow N ,\ \iota_i:T_i\hookrightarrow N,\ \zeta_i:D_i\hookrightarrow N,\quad i=1,\cdots,k\] the inclusion maps to the ambient manifold $N$.
Our strategy is as follows: cut $N$ along all JSJ-tori and all tori $\{T_1,\cdots,T_k\}$ that appears in each Seifert piece of the JSJ-decomposition of $N$ as above. By Lemma \ref{TorusTorsionFormula}, the JSJ-tori do not contribute to the $L^2$-torsion. Then by the sum formula of Lemma \ref{BasicProppertyTorsion}, we have the following formula:
\begin{equation}\label{GiantEquation}
 \tautwo(N,\rho)=
 \prod_{M\subset  N \text{ is a  Seifert piece}}\frac{\tautwo(\Sigma\times S^1,\rho\circ\kappa_*;\kappa_*)\prod_{i=1}^k\tautwo(D_i,\rho\circ \zeta_{i*};\zeta_{i*})}{\prod_{i=1}^k\tautwo(T_i,\rho\circ \iota_{i*};\iota_{i*})}
\end{equation}
It remains to calculate the terms appearing in Theorem \ref{GiantEquation}.

Firstly the easiest part. Since $\iota_{i*}(\pi_1(T_i))$ has infinite order in $G$ then the twisted $L^2$-torsion of the admissible triple $(T_i,\rho\circ \iota_{i*};\iota_{i*})$ is trivially $1$ by Lemma \ref{TorusTorsionFormula}.

We now compute $\tautwo(\Sigma\times S^1,\rho\circ\kappa_*;\kappa_*)$. Set $\pi:=\pi_1(\Sigma\times S^1)$, the CW chain complex of the universal cover $\widehat{\Sigma\times S^1}$ is 
\[C_*(\widehat{\Sigma\times S^1})=(0\longrightarrow \C \pi^{2g+b+k-1}\stackrel{\partial_2}\longrightarrow \C\pi^{2g+b+k} \stackrel{\partial_1}\longrightarrow \C \pi \stackrel{\partial_0} \longrightarrow 0 )\]
in which
\[\partial_2=\begin{pmatrix}
1-h & 0 & \cdots & 0 & *\\
0 & 1-h &  & \vdots & \vdots\\
\vdots & &\ddots & 0 & *\\
0 &\cdots & 0 & 1-h & *
\end{pmatrix},\quad \partial_1=\begin{pmatrix}
*\\ \vdots \\ * \\ 1-h
\end{pmatrix}.
\]
Then the $L^2$-chain complex of $\Sigma\times S^1$ twisted by $(\pi,\rho\circ\kappa_*;\kappa_*)$ is
\[C_*^{(2)}(\Sigma\times S^1,\rho\circ\kappa_*;\kappa_*)=(0\longrightarrow l^2(G)^{2g+b+k-1}\stackrel{\partial_2^\rho}\longrightarrow l^2(G)^{2g+b+k} \stackrel{\partial_1^\rho}\longrightarrow l^2(G) \stackrel{\partial_0} \longrightarrow 0 )\]
in which

\[\partial_2^\rho=\begin{pmatrix}
I^{n\times n}-h\rho(h) & 0 & \cdots & 0 & *\\
0 & I^{n\times n}-h\rho(h) &  & \vdots & \vdots\\
\vdots & &\ddots & 0 & *\\
0 &\cdots & 0 & I^{n\times n}-h\rho(h) & *
\end{pmatrix},\quad \partial_1^\rho=\begin{pmatrix}
*\\ \vdots \\ * \\I^{n\times n}-h\rho(h)
\end{pmatrix}.
\]
We have identified  $h$ with its image under $\kappa_*$ in $\pi_1(N)=G$ for notational convenience. If the modulus of all eigenvalues of $\rho(h)$ are $\lambda_1,\cdots,\lambda_n$, by properties of regular Fuglede-Kadison determinant and Lemma \ref{MahlerMasureFormula}, \ref{MatrixCalculationTorsion}, we know that
\[\begin{aligned}
 \tautwo(\Sigma\times S^1,\rho\circ\kappa_*;\kappa_*)&=\rdet G(I^{n\times n}-h\rho(h))^{2g+b+k-2}\\
 &=\operatorname{Mah}(\prod_{r=1}^n(1-z\lambda_r))^{2g+b+k-2}\\
 &=\Lambda^{-(2g+b+k-2)}.
\end{aligned}\]

Then we compute $\tautwo(D_i,\rho\circ \zeta_{i*};\zeta_{i*})$. It is easy to see that the generator of $\pi_1(D_i)$ is represented by $h^{m_i}f_i^{n_i}$, where $(m_i,n_i)$ is a pair of integers such that $m_ip_i-n_iq_i=1$. Then we have
\[\tautwo(D_i,\rho\circ \zeta_{i*};\zeta_{i*})=\rdet G(I^{n\times n}-h^{m_i}f_i^{n_i}\cdot\rho(h^{m_i}f_i^{n_i}))^{-1}\]
where $h,f_i$ are again viewed as elements in $G$. Since $h$ and $f_i$ commute and are simultaneously upper triangularisable,  then the modulus of all eigenvalues of $\rho(h^{m_i}f_i^{n_i})$ are $\lambda_1^{1/p_i},\cdots,\lambda_n^{1/p_i}$. Note that $h^{m_i}f_i^{n_i}$ is an infinite order element, by Lemma \ref{MahlerMasureFormula} we have
\[\rdet G(I^{n\times n}-h^{m_i}f_i^{n_i}\cdot\rho(h^{m_i}f_i^{n_i}))=\operatorname{Mah}(\prod_{r=1}^n(1-z\lambda_r^{1/p_i}))=\Lambda^{-1/p_i},\]
and then $\tautwo(D_i,\rho\circ \zeta_{i*};\zeta_{i*})=\Lambda^{1/p_i}$.

Finally, combining the calculations above, we have
\[\begin{aligned}
 &\quad\frac{\tautwo(\Sigma\times S^1,\rho\circ\kappa_*;\kappa_*)\prod_{i=1}^k\tautwo(D_i,\rho\circ \zeta_{i*};\zeta_{i*})}{\prod_{i=1}^k\tautwo(T_i,\rho\circ \iota_{i*};\iota_{i*})}\\
 &=\Lambda^{-(2g+b+k-2)+\sum_{i=1}^k\frac1{p_i}}\\
 &=\Lambda^{2-2g-b-\sum_{i=1}^k(1-\frac1{p_i})}\\
 &=\Lambda^{\chi_{\operatorname{orb}}}.
\end{aligned}\]
And the conclusion follows from Theorem \ref{GiantEquation}.
\end{proof}

\subsection{Twisted $L^2$-torsion for hyperbolic or mixed manifolds}

In this part, we assume that $N$ is not a graph manifold, or equivalently, $N$ contains at least one hyperbolic piece in its geometrization decomposition. Then $N$ is either hyperbolic or so-called mixed. By Agol's RFRS criterion for virtual fibering \cite{agol2008criteria} and the virtual specialness of 3-manifolds having at least one hyperbolic piece \cite{agol2013virtual, przytycki2018mixed}, we can assume that $N$ has a regular finite cover that fibers over circle. 

For future convenience, we introduce the following notions.

\begin{defn}\label{AlexanderDefinition}
Let $G$ be a finitely generated, residually finite group. For any cohomology class $\psi\in H^1(G;\R)$, and any real number $t>0$, there is an 1-dimensional representation
\[\psi_t:G\ra \C^\times,\quad g\mapsto t^{\psi(g)}.\]
This representation can be used to twist $\C G$, determining a $\C G$-homomorphism: \[\kappa(\psi,t):\C G\ra \C G,\quad g\mapsto t^{\psi(g)}g,\ g\in G\]
and extend $\C$-linearly. The $\C G$-homomorphism $\kappa(\psi,t)$ is called the \emph{Alexander twist of $\C G$ associated to $(\psi,t)$}.
\end{defn}
\begin{defn}
A positive function $f:\R^+\ra \R^+$ is multiplicatively convex if the function \[F:\R\ra \R,\quad t\longmapsto\log f(e^t)\]
is a convex function. In particular, a multiplicatively convex function is continuous and everywhere positive.
\end{defn}
Our main technical tool is the following theorem due to Liu \cite[Theorem 5.1]{liu2017degree}.
\begin{thm}\label{LiuToolTheorem}
Let $G$ be a finitely generated, residually finite group. For any square matrix $A$ over $\C G$ and any 1-cohomology class $\psi\in H^1(G;\R)$, the function
\[t\longmapsto \rdet G (\kappa(\psi,t)A),\quad t>0\]
is either constantly zero or multiplicatively convex (and in particular every where positive).
\end{thm}

With the above preparations, we are now ready to prove Theorem \ref{TheoremA} for hyperbolic or mixed 3-manifolds.
\begin{thm}\label{TorsionForHyper}
Suppose $N$ is a compact orientable irreducible 3-manifold with empty or incompressible toral boundary. Assume that $N$ is hyperbolic or mixed. Then $\tautwo(N,\rho)>0$. 
\end{thm}
\begin{proof}
Since twisted $L^2$-torsion behaves multiplicatively with respect to finite covers by Lemma \ref{BasicProppertyTorsion}, we may assume without loss of generality that $N$ itself fibers over circle.

The following procedure is analogous to  \cite[Theorem 8.5]{dubois2016l2}. Denote by $\Sigma$ a fiber of $N$, and $f:\Sigma\ra\Sigma$ the monodromy such that $N$ is homeomorphic to the mapping torus
\[T_f(N)=\Sigma\times [-1,1]/(x,-1)\sim(f(x),1).\]
We can assume by isotopy that $f$ has a fixed point $P$. Construct a CW structure $X$ modeled on $\Sigma$ with a single 0-cell $P$, $k$ 1-cells $E_1,\cdots,E_n$, and a 2-cell $\sigma$. By CW approximation, there is a cellular map $g:\Sigma\ra \Sigma$ homotopic to $f$. Then the mapping torus $T_g(\Sigma)$ is homotopy equivalence to $N$, which is a simple homotopy equivalent since the Whitehead group of a fibered 3-manifold is trivial, see \cite[Theorem 19.4, Theorem 19.5]{waldhausen1978algebraicPart2}. Hence by Lemma \ref{BasicProppertyTorsion} we have
\[\tautwo(N,\rho)=\tautwo(T_g(\Sigma),\rho).\]
We proceed to describe a CW complex for the mapping torus $T_g(\Sigma)$. Suppose $\pi_1(N)=\pi_1(T_g(\Sigma))=G$. The cells in each dimensions are
\[\{\sigma\times I\},\ \{\sigma,E_1\times I,\cdots,E_k\times I\},\ \{E_1,\cdots,E_k,P\times I\},\ \{P\}\]
where $I=[-1,1]$. Let $e_i:=[E_i]\in G,\  h:=[P\times I]\in G$ be the fundamental group elements represented by the corresponding loops. Denote by $\psi\in H^1(G;\R)$ the 1-cohomology class dual to the fiber $\Sigma$, then we have
\[\psi(h)=1,\quad \psi(e_1)=\cdots=\psi(e_k)=0.\]
The CW chain complex of $\widehat{T_g(\Sigma)}$ has the form
\[C_*(\widehat{T_g(\Sigma)})=(0\longrightarrow \C G\stackrel{\partial_3}\longrightarrow \C G^{k+1} \stackrel{\partial_2}\longrightarrow \C G^{k+1} \stackrel{\partial_1} \longrightarrow \C G\stackrel{\partial_0} \longrightarrow 0 )\]
in which
\[\partial_3=(1-h,*,\cdots,*),\quad \partial_2=\begin{pmatrix}
* & *\\
I^{k\times k}-h\cdot A & *
\end{pmatrix},\quad\partial_1=\begin{pmatrix}
*\\
1-h
\end{pmatrix}
\]
and $``*"$ stands for matrices of appropriate size, $A$ is a matrix over $\C[\ker\psi]$ of size $k\times k$. Denote by $A_\rho$ the matrix $A$ twisted by $\rho$, as in Proposition \ref{MatrixRepPartialTwist}, then the $L^2$-chain complex of $T_g(\Sigma)$ twisted by $(G,\rho;\operatorname{id}_G)$ is
\[C_*^{(2)}(T_g(\Sigma),\rho)=(0\longrightarrow l^2(G)^n\stackrel{\partial_3^\rho}\longrightarrow l^2(G)^{n(k+1)} \stackrel{\partial_2^\rho}\longrightarrow l^2(G)^{n(k+1)} \stackrel{\partial_1^\rho} \longrightarrow l^2(G)^n\stackrel{} \longrightarrow 0 )\]
in which
\[\partial_3^\rho=(I^{n\times n}-h\rho(h),*,\cdots,*),\quad \partial_2^\rho=\begin{pmatrix}
* & *\\
I^{nk\times nk}-h\cdot\rho(h)A_\rho & *
\end{pmatrix},\quad\partial_1^\rho=\begin{pmatrix}
*\\
I^{n\times n}-h\rho(h)
\end{pmatrix}.
\]
Consider the following two matrices
\[S:=I^{n\times n}-h\rho(h),\quad T:=I^{nk\times nk}-h\rho(h)A_\rho\]
and the matrices under the Alexander twist associated to $(\psi,t)$:
\[S(t):=\kappa(\psi,t)S=I^{n\times n}-t\cdot h\rho(h),\quad T(t):=\kappa(\psi,t)T=I^{nk\times nk}-t\cdot h\rho(h)A_\rho.\]
For any real number $t>0$ sufficiently small, the two matrices $S(t)$ and $T(t)$ are both invertible with regular Fugelede-Kadison determinant equal to 1, see \cite[Proposition 8.8]{dubois2016l2}. Then Liu's Theorem \ref{LiuToolTheorem} applies to show that these two Fugelede-Kadison determinants are positive when $t=1$. It follows from Theorem \ref{MatrixCalculationTorsion} that $\tautwo(N,\rho)=\rdet G T(1)\cdot \rdet G S(1)^{-2}$ is positive.
\end{proof}

Theorem \ref{TheoremA} then follows from Theorem \ref{TorsionForGraph} and Theorem \ref{TorsionForHyper}.

\section{Continuity of twisted $L^2$-torsion on representation varieties}\label{section 5}
Let $N$ be any compact orientable irreducible 3-manifold with empty or incompressible toral boundary, set $G:=\pi_1(N)$. Suppose that $G$ is infinite, and denote by $\mathcal R_n(G):=\operatorname{Hom}(G,\slC n)$ the representation variety, then Theorem \ref{TheoremA} implies that the twisted $L^2$-torsion can be viewed as a positive function
\[\rho\longmapsto\tautwo(N,\rho),\ \rho\in \mathcal R_n(G).\]
The continuity of this torsion function is an interesting but rather hard question. The work of Liu \cite[Theorem 1.2]{liu2017degree} have shown that the torsion function is continuous in $\operatorname{Hom}(G,\R)$ along the Alexander twists, we remark that in his article the twist is not unimodular, and an equivalence class for torsion functions is introduced to guarantee well-definedness. If $N$ is hyperbolic and $\rho_0:G\ra \operatorname{PSL(2,\C)}$ is a holonomy representation associated to the hyperbolic structure, and $\rho\in\mathcal{R}_2(G)$ is a lifting of $\rho_0$ (such lifting always exists, see \cite[Corollary 2.2]{culler1986lifting}), then Bernard and Raimbault \cite{benard2022twisted} proved that the torsion function is analytic near $\rho$. The continuity of the torsion function in general is wide open. In this section we present a partial result on the continuity of the twisted $L^2$-torsion function, namely Theorem \ref{TheoremB}. We start with a brief discussion of the $L^2$-Alexander torsions since it is closely related to the proof of Theorem \ref{TheoremB}.

\subsection{$L^2$-Alexander torsions}
The $L^2$-torsion twisted by 1-dimensional representations are called \emph{the $L^2$-Alexander torsion}. To be precise, for any 1-cohomology class $\psi\in H^1(G;\R)$ and any real number $t>0$, the \emph{$L^2$-Alexander torsion} of $N$ associated to $(\psi,t)$ is defined to be
\[A^{(2)}(N,\psi,t):=\tautwo (C^{(2)}_*(N,\psi_t)).\]
Recall that $\psi_t:G\ra \C^\times$ maps $g\in G$ to $t^{\psi(g)}$ is the representation associated to $(\psi,t)$. Since $\psi_t$ is not a unimodular representation, the $L^2$-Alexander torsion depends on the based $\C G$-chain complex $C_*(\widehat N)$. Indeed, altering the $\C G$-basis of $C_*(\widehat N)$, the base change matrix for $C_*^{(2)}(N,\psi_t)$ will be a permutation matrix with entries $\pm t^{\pm\psi(g_i)}g_i$ (compare Proposition \ref{InvarianceOfBasis}), whose regular Fuglede-Kadison determinant is $t^{\sum_i \pm\psi(g_i)}$. Since $g_i\in G$ are independent of $\psi$ and $t$, the continuity of $A^{(2)}(N,\psi,t)$ as a function of $(\psi,t)\in H^1(G;\R)\times \R_+$ is irrelevant of the choice of cellular basis, here $H^1(N;\R)$ is given the usual real vector space topology.

In literature 
\cite{dubois2015l2, dubois2016l2}, one consider $A^{(2)}(N,\psi,t)$ as a function of $t$, and introduce an equivalence relation between functions. Namely, two functions $f_1,f_2:\R_+\ra [0,+\infty)$ are equivalent if and only if there exists a real number $r$ such that 
\[f_1(t)=t^r\cdot f_2(t)\]
holds for all $t>0$. In this case we denote by $f_1\dot{=} f_2$. So the equivalence class of $A(N,\psi,t)$ as a function of $t$ does not depend on the choice of cellular basis.

Another way to cure the ambiguity is to modify $\psi_t$ to be a unimodular 2-dimensional representation. Set \[\psi_t\oplus\psi_{t^{-1}}:G\ra\slC2,\quad g\mapsto\begin{pmatrix}
t^{\psi(g)} & 0\\
0 & t^{-\psi(g)}
\end{pmatrix}.\]
Then it is easy to observe that $C^{(2)}_*(N,\psi_t\oplus\psi_{t^{-1}})=C^{(2)}_*(N,\psi_t)\oplus C^{(2)}_*(N,\psi_{t^{-1}})$ and hence by L\"uck \cite[Theorem 3.35]{luck2002l2} we have
\[A^{(2)}(N,\psi,t)\cdot A^{(2)}(N,\psi,t^{-1})=\tautwo(N,\psi_t\oplus\psi_{t^{-1}})\]
which does not depend on the choice of cellular basis. This fact motivates the following definition.
\begin{defn}
For any $\psi\in H^1(G;\R)$ and $t>0$, we define the \emph{symmetric $L^2$-Alexander torsion of $N$ associated to $(\psi,t)$} to be \[A^{(2)}_{\operatorname{sym}}(N,\psi,t):=\tautwo(N,\psi_t\oplus\psi_{t^{-1}})^{\frac12}.\]
\end{defn}
It is shown in \cite[Chapter 6]{dubois2016l2} that the $L^2$-Alexander torsion satisfies
\[A^{(2)}(N,\psi,t)=t^{-\psi(c_1(e))}\cdot A^{(2)}(N,\psi,t^{-1})\]
where $c_1(e)\in H_1(N;\Z)$ is independent of $(\psi,t)$. This shows that \[A^{(2)}_{\operatorname{sym}}(N,\psi,t)=t^r\cdot A^{(2)}(N,\psi,t)\]
for some real number $r$.
We remark that, as a function of $(\psi,t)$, the continuity of $A^{(2)}(N,\psi,t)$ defined by any CW structure is equivalent to the continuity of $A^{(2)}_{\operatorname{sym}}(N,\psi,t)$.

As an illustration of the various definitions, we rediscover the $L^2$-Alexander torsion $A^{(2)}(N,\psi,t)$ for graph manifold $N$ using Theorem \ref{TorsionForGraph}. The calculation is first carried out by Herrmann \cite{herrmann20162} for Seifert fibering space and by Dubois et al. \cite{dubois2016l2} for graph manifolds.
\begin{thm}\label{TwistiedLtwoAlexanderGraph}
Let $N$ be a graph manifold with infinite fundamental group. Suppose that $N\not= S^1\times D^2$ and $N\not=S^1\times S^2$. Then a representative of the $L^2$-torsion twisted by $(\psi,t)$ is \[A^{(2)}(N,\psi,t)=\max\{1,t^{x_N(\psi)}\}\]
where $x_N$ is the Thurston norm for $H^1(N;\R)$.
\end{thm}
\begin{proof}
For $t\geqslant1$, set $\rho:=\psi_t\oplus \psi_{t^{-1}}$, then by Theorem \ref{TorsionForGraph}, we have
\[A^{(2)}_{\operatorname{sym}}(N,\psi,t)^2=\tautwo(N,\psi_t\oplus \psi_{t^{-1}})=\prod_{M\subset  N \text{ is a  Seifert piece}} t^{-|\psi(h)|\cdot\chi_{\operatorname{orb}}}\]
where $h\in H^1(M;\R)$ is represented by the regular fiber of $M$ and $\chi_{\operatorname{orb}}$ is the orbifold Euler characteristic of $M/S^1$. By our assumption on $N$, we know that $\chi_{\operatorname{orb}}\leqslant 0$, so $-|\psi(h)|\cdot\chi_{\operatorname{orb}}=x_M(\psi)$ by \cite[Lemma A]{herrmann20162}, where $x_M$ is the Thurston norm for $H^1(M;\R)$. Then by \cite[Proposition 3.5]{eisenbud1985three}, we have
\[\sum_{M\subset  N \text{ is a  Seifert piece}}x_M(\psi)=x_N(\psi)\]
and then
\[A^{(2)}_{\operatorname{sym}}(N,\psi,t)^2=t^{x_N(\psi)},\quad t\geqslant1.\]
Since the symmetric $L^2$-Alexander torsion is by definition symmetric, so
\[A^{(2)}_{\operatorname{sym}}(N,\psi,t)=\max\{t^{\frac12 x_N(\psi)},t^{-\frac12 x_N(\psi)}\}\dot{=} \max\{1,t^{x_N(\psi)}\}.\]

\end{proof}

It follows that the $L^2$-Alexander torsion of graph manifolds is continuous in $(\psi,t)\in H^1(G;\R)\times\R^+$. For a general 3-manifold $N$, the continuity of the $L^2$-Alexander torsion is a hard question. Liu \cite{liu2017degree} and L\"uck \cite{luck2018twisting} independently proved that the $L^2$-Alexander torsion function is always positive. Moreover Liu proved in the same article that $A^{(2)}(N,\psi,t)$ is multiplicatively convex with respect to $t$, and in particular it is continuous. L\"uck \cite[Chapter 10]{luck2018twisting} conjectured that this function is continuous with respect to $(\psi,t)\in H^1(N;\R)\times \R^+$. We will see that this statement is true.
\begin{thm}\label{ContwrtCohomology}
Let $N$ be a compact orientable irreducible 3-manifold with empty or incompressible toral boundary. Suppose $\pi_1(N)=G$ is infinite. Then any representative of the $L^2$-Alexander torsion function $A^{(2)}(N,\psi,t)$ is continuous with respect to $(\psi,t)\in H^1(N;\R)\times \R^+$.
\end{thm}

Theorem \ref{TheoremB} is now a  corollary of Theorem \ref{ContwrtCohomology}, as we restate here
\begin{thm}\label{ContinuityOfAbelianReps}
Let $N$ be a compact orientable irreducible 3-manifold with empty or incompressible toral boundary. Suppose $\pi_1(N)=G$ is infinite. Define  $\mathcal R^{\operatorname{t}}_n(G)$ to be the subvariety of $\mathcal R_n(G)$ consisting of upper triangular representations. Then the twisted $L^2$-torsion function
\[\rho\longmapsto \tautwo(N,\rho)\]
is continuous with respect to $\rho\in \mathcal R^{\operatorname{t}}_n(G)$.

\end{thm}

\begin{proof}
Fix a CW structure for $N$ and fix a choice of cell-lifting to $\widehat N$, so we can talk about the $L^2$-Alexander torsion unambiguously. For any $\rho\in \mathcal R^{\operatorname{t}}_n(G)$, we can assume that
\[\rho(g)=\begin{pmatrix}
\chi_1(g) & \cdots & *\\
 & \ddots & \vdots \\
 & & \chi_n(g)
\end{pmatrix}\]
where $\chi_k:G\ra \C^\times$ are characters. The modulus of those characters can be written as
\[|\chi_k|=e^{\phi_k},\quad g\longmapsto e^{\phi_k(g)}\]
for some real 1-cohomology class $\phi_k\in H^1(G;\R)$. The classes $\phi_1,\cdots,\phi_n$ are continuous with respect to $\rho\in \mathcal R^{\operatorname{t}}_n(G)$.

Let $V_n$ be the $G$-invariant subspace of $V$ corresponding to $\chi_n$, and let $V':=V/V_n$, then there is an exact sequence of $G$-representations
\[0\longrightarrow V_n\longrightarrow V\longrightarrow V'\longrightarrow 0\]
where the $G$-actions are given by
\[\rho_n(g)=\chi_n(g),\quad \rho(g)=\begin{pmatrix}
\chi_1(g) & \cdots & *\\
 & \ddots & \vdots \\
 & & \chi_n(g)
\end{pmatrix},\quad \rho'(g)=\begin{pmatrix}
\chi_1(g) & \cdots & *\\
 & \ddots & \vdots \\
 & & \chi_{n-1}(g)
\end{pmatrix}\]
respectively. Then by L\"uck \cite[Lemma 3.3]{luck2018twisting}, we have
\[\tautwo(N,\rho)=\tautwo(N,\rho_n)\tautwo(N,\rho').\]
Since unitary twists have no effects on $L^2$-torsions by L\"uck \cite[Theorem 4.1]{luck2018twisting}, we have \[\tautwo(N,\rho_n)=\tautwo(N,e^{\phi_n})=A^{(2)}(N,\phi_n,e).\] The above process can then be applied to $\rho'$ and finally we have the formula
\[\tautwo(N,\rho)=A^{(2)}(N,\phi_1,e)\cdots A^{(2)}(N,\phi_n,e).\]
Since the cohomology classes $\phi_1,\cdots,\phi_n$ vary continuously with respect to $\rho\in\mathcal R^{\operatorname{t}}_n(G)$, the conclusion follows from Theorem \ref{ContwrtCohomology}.
\end{proof}

The following part of this section is devoted to the proof of Theorem \ref{ContwrtCohomology}. We will need the notion of Alexander multi-twists.
\subsection{Alexander multi-twists of matrices}
Recall that $G$ is any finitely generated, residually finite group. For any collection of 1-cohomology classes $\Phi=(\phi_1,\cdots,\phi_n)\in \prod_{i=1}^n H^1(G;\R)$ and any collection of positive real numbers $T=(t_1,\cdots,t_n)\in \R^n_+$, we define a $\C G$-homomorphism
\[\kappa(\Phi,T):\C G\ra\C G,\quad g\ra t_1^{\phi_1(g)}\cdots t_n^{\phi_n(g)}\cdot g,\ g\in G.\]
This is called the \emph{Alexander multi-twist of $\C G$ associated to $(\Phi,T)$}. 

\begin{prop}\label{PropertyOfMultiTwist} Basic properties of the Alexander multi-twist:

(1) (Associativity) Suppose
$\Phi=(\phi_1,\cdots,\phi_n),\ T=(t_1,\cdots,t_n).$
Then
\[\kappa(\Phi,T)=\kappa(\phi_1,t_1)\circ\cdots\circ\kappa(\phi_n,t_n).\]

(2) (Commutativity) $\kappa(\phi_1,t_1)\circ\kappa(\phi_2,t_2)=\kappa(\phi_2,t_2)\circ\kappa(\phi_1,t_1)$.

(3) (Change of coordinate) Let $r_1, r_2\in\R$, then we have
\[\kappa(r_1\phi_1+r_2\phi_2,t)=\kappa(\phi_1,t^{r_1})\circ\kappa(\phi_2,t^{r_2}).\]
\[\kappa(\phi,t_1^{r_1}t_2^{r_2})=\kappa(r_1\phi,t_1)\circ\kappa(r_2\phi,t_2).\]
\end{prop}

The Alexander multi-twist extends to an endomorphism of the matrix algebra with entries in $\C G$.

In the following part of this section, we shall fix a square matrix $\Omega$ over $\C G$, and suppose that $\rdet G(\Omega)$ is not zero. For any collection of 1-cohomology classes $\Phi=(\phi_1,\cdots,\phi_n)$ and positive real numbers $T=(t_1,\cdots,t_n)$, we introduce the notation
\[V_{\Phi}(T):=\rdet G (\kappa(\Phi,T)\Omega).\]

\begin{prop}\label{coordinate-convex}
For any fixed choice of $\Phi$, the multi-variable function $V_\Phi(T)$ is everywhere positive and is multiplicatively convex in each coordinate with respect to $T=(t_1,\cdots,t_n)\in \R^n_+$.
\end{prop}

\begin{proof}
By associativity and commutativity of Alexander multi-twist we have 
\[\kappa(\Phi,T)\Omega=\kappa(\phi_i,t_i)\circ\kappa(\Phi',T')\Omega\]
where $(\Phi',T')$ are variables other than $(\phi_i,t_i)$. The conclusion then follows from applying Theorem \ref{LiuToolTheorem} to each $i$.
\end{proof}

\begin{thm}\label{MultivariableConvexityOfDeterminant}
For any fixed choice of $\Phi$, the multi-variable real function $V_\Phi(T)$ is multiplicatively convex with respect to $T=(t_1,\cdots,t_n)\in \R^n_+.$
\end{thm}

\begin{proof} We will prove that for any fixed choice of $\Phi$ and every positive integer $k\leqslant n$, the function $V_\Phi(T)$ is multiplicatively convex with respect to the first $k$ coordinates.

The case $k=1$ is proved by Proposition \ref{coordinate-convex}. Assume the claim holds for $(k-1)$ and consider \[V_{\phi_1,\cdots,\phi_{k}}(t_1,\cdots,t_{k})=V_\Phi(T)\]
as a function of the first $k$ variables of $\Phi$ and $T$. It suffices to prove that for any $\theta\in(0,1)$ and any collection of positive numbers $r_1,\cdots,r_k>0$, $s_1,\cdots,s_k>0$, then
\[\left(V_{\phi_1,\cdots,\phi_{k}}(r_1,\cdots,r_k)\right)^\theta\cdot\left(V_{\phi_1,\cdots,\phi_{k}}(s_1,\cdots,s_k)\right)^{1-\theta}\geqslant V_{\phi_1,\cdots,\phi_{k}}(r_1^\theta s_1^{1-\theta},\cdots,r_k^\theta r_k^{1-\theta}).\]
We can assume that $r_1\not= s_1$, otherwise this inequality degenerates to the $(k-1)$ case after permuting the coordinates. Consider $\psi_1=\phi_1+\lambda\phi_k$ for a real number $\lambda$ which will be determined later. We have the identity that for all $t_1,\cdots,t_k>0$,
\[
V_{\psi_1,\phi_2,\cdots,\phi_k}(t_1,\cdots,t_{k-1},t_k)=V_{\phi_1,\phi_2,\cdots,\phi_{k}}(t_1,\cdots,t_{k-1},t_1^{\lambda}t_k).\]
By induction hypothesis, for all $r>0$, we have
\[
\begin{aligned}
 \left(V_{\psi_1,\phi_2,\cdots,\phi_{k}}(r_1,\cdots,r_{k-1},r)\right)^\theta&\cdot\left(V_{\psi_1,\phi_2,\cdots,\phi_{k}}(s_1,\cdots,s_{k-1},r)\right)^{1-\theta}\\
 &\geqslant V_{\psi_1,\phi_2,\cdots,\phi_{k}}\left(r_1^\theta s_1^{1-\theta},\cdots,r_{k-1}^\theta s_{k-1}^{1-\theta},r\right)
\end{aligned}
\]
which is equivalent to
\[
\begin{aligned}
\left(V_{\phi_1,\cdots,\phi_{k}}(r_1,\cdots,r_{k-1},r_1^\lambda r)\right)^\theta&\cdot\left(V_{\phi_1,\cdots,\phi_{k}}(s_1,\cdots,s_{k-1},s_1^\lambda r)\right)^{1-\theta}\\&\geqslant V_{\phi_1,\cdots,\phi_{k}}\left(r_1^\theta s_1^{1-\theta},\cdots,r_{k-1}^\theta s_{k-1}^{1-\theta},(r_1^\lambda r)^\theta\cdot (s_1^\lambda r)^{1-\theta}\right).
\end{aligned}
\]
Since $r_1\not=s_1$, we can prescribe $\lambda\in\R$ and $r>0$ by solving the following equations
\[r_1^\lambda r=r_k,\quad s_2^\lambda r=s_k.\]
This finishes the induction.
\end{proof}

\begin{cor}\label{convex}
For any fixed $(\Phi,T)\in \prod_{i=1}^n H^1(G;\R)\times\R^n_+$, the function $W_{\Phi,T}:\R^n\ra \R$,
\[W_{\Phi,T}(s_1,\cdots,s_n):=\log\left(V_{s_1\phi_1,\cdots,s_n\phi_s}(T)\right)\]
is convex. In particular it is continuous.
\end{cor}
\begin{proof}
This follows from the identity
\[W_{\Phi,T}(s_1,\cdots,s_n):=\log\left(V_{s_1\phi_1,\cdots,s_n\phi_s}(T)\right)=\log\left(V_\Phi(t_1^{s_1},\cdots,t_n^{s_n})\right)\]
and the multiplicatively convexity of $V_\Phi(T)$.
\end{proof}

\begin{thm}\label{ContMatrixCohomology}
The regular Fuglede-Kadison determinant map
$\rdet G(\kappa(\phi,t)\Omega)$ is continuous with respect to $(\phi,t)\in H^1(G;\R)\times\R_+$.
\end{thm}
\begin{proof}
Let $\Psi=(\psi_1,\cdots,\psi_k)$ be a basis for the real vector space $H^1(G;\R)$. Suppose
\[\phi=\sum_{i=1}^k c_j \psi_j,\quad 1\leqslant i\leqslant n,\]
where the coefficients $c_j$ are continuous with respect to $\phi\in H^1(G;\R)$. Then
\[
\begin{aligned}
\kappa(\phi,t)\Omega&= \kappa(c_1\psi_1,t)\circ \cdots \circ\kappa(c_k\psi_k,t)\Omega\\
&= \kappa(c_1\log t \cdot\psi_1,e)\circ \cdots \circ\kappa(c_k\log t\cdot \psi_k,e)\Omega\\
&=\kappa\Big((c_1\log t\cdot\psi_1,\cdots,c_k\log t\cdot\psi_k),(e,\cdots,e)\Big)\Omega.
\end{aligned}\]
By definition we have 
\[\rdet G(\kappa(\phi,t)\Omega)=\exp W_{\Psi,(e,\cdots,e)}(c_1\log t,\cdots,c_k\log t).\]
The continuity follows from corollary \ref{convex}.
\end{proof}

\subsection{Applications to 3-manifolds}
We are now ready to prove Theorem \ref{ContwrtCohomology}.

\begin{proof}[{Proof of Theorem \ref{ContwrtCohomology}}]
If $N$ is a graph manifold, then Theorem \ref{TwistiedLtwoAlexanderGraph} offers an explicit formula for the $L^2$-Alexander torsion, the theorem holds since the Thurston norm is continuous in $H^1(N;\R)$.

If $N$ is a compact connected orientable irreducible 3-manifold which is hyperbolic or mixed, then as in the proof of Theorem \ref{TorsionForHyper}, we can find a regular finite covering $p:\widetilde N\ra N$ of degree $d$. Since by Lemma \ref{BasicProppertyTorsion} we have \[\tautwo(N,\psi_t\oplus\psi_{t^{-1}})^d=\tautwo(\widetilde N,p^*\psi_t\oplus p^*\psi_{t^{-1}}),\]
and then $A^{(2)}_{\operatorname{sym}}(N,\psi,t)^d=A^{(2)}_{\operatorname{sym}}(\widetilde N,p^*\psi,t)$. Note that the pullback map $p^*:H^1(N;\R)\ra H^1(\widetilde N;\R)$ is a continuous embedding, we only need to prove the theorem for $\widetilde N$. So we can assume without loss of generality that our manifold $N$ fibers over circle. From proof of Theorem \ref{TorsionForHyper} we see that
\[A^{(2)}(N,\psi,t)=\rdet G  (\kappa(\psi,t)T)\cdot \rdet G(\kappa(\psi,t)S)^{-2}\]
where $T=I^{k\times k}-hA_\rho$, $S=1-h$ are square matrices over $\C G$ with positive regular Fuglede-Kadison determinant. The conclusion follows immediately from Theorem \ref{ContMatrixCohomology}.
\end{proof}

The continuity result can be used to improve the calculation of the $L^2$-Alexander torsion associated to fibered classes.

\begin{thm}
Let $N$ be any compact, connected, irreducible, orientable 3-manifold with empty or incompressible toral boundary. Suppose $\pi_1(N)$ is infinite, $N\not=S^1\times D^2$ and $N\not=S^1\times S^2$. Let $\phi\in H^1(N;\R)$ be in the interior of a fibered cone. Then there exists a representative of $L^2$-Alexander torsion associated to $(\phi,t)$ such that
\[A^{(2)}(N,\phi,t)=\left\{
\begin{aligned}
 &1,\quad t<\frac1{h(\phi)},\\
 &t^{x_N(\phi)},\quad t>h(\phi)
\end{aligned}\right.\]
where $h(\phi)$ is the entropy function defined on the fibered cone of $H^1(N;\R)$ (compare \cite[Section 8]{dubois2015l2}).
\end{thm}
\begin{proof}
Let $\phi_n\in H^1(G;\Q)$ be a sequence in the fibered cone that converge to $\phi$.
By \cite[Theorem 8.5]{dubois2015l2}, for any $n$ we have
\[A^{(2)}(N,\phi_n,t)=\left\{
\begin{aligned}
 &1,\quad t<\frac1{h(\phi_n)},\\
 &t^{x_N(\phi_n)},\quad t>h(\phi_n).
\end{aligned}\right.\]
By Theorem \ref{ContwrtCohomology} we have
\[A^{(2)}(N,\phi_n,t)\ra A^{(2)}(N,\phi,t),\quad n\ra\infty\]
for any $t\in\R$.
Since the entropy and the Thurston norm are continous functions of $H^1(N;\R)$, we have
\[h(\phi_n)\ra h(\phi),\quad x_N(\phi_n)\ra x_N(\phi),\quad n\ra \infty.\]
This proves our claim.
\end{proof}

\bibliography{ref.bib}
\end{document}